\documentclass[a4paper,12pt]{article}
\usepackage{amssymb}
\usepackage{mathrsfs}
\usepackage{amsmath}
\usepackage{amsfonts,amsthm,amssymb}
\usepackage{amsfonts}
\usepackage{graphics}
\usepackage{color}
\usepackage{graphicx}
\usepackage{epstopdf}
\usepackage{subfigure}
\usepackage{appendix}

\usepackage{dsfont}
\usepackage{float}
\usepackage{booktabs}
\usepackage{array}
\usepackage{diagbox}
\usepackage{multirow}

\usepackage{geometry}
\newtheorem{lemma}{Lemma}[section]
\newtheorem{theorem}{Theorem}[section]

\newtheorem{definition}{Definition}[section]

\begin{document}
\title{Extremal spectral radius of weighted adjacency matrices of bicyclic graphs\footnote{This paper was supported by the National Natural Science Foundation of China [No. 11971406]. }}
\author{\small Jiachang Ye$^{1}$, Junli Hu$^{2}$,  Xiaodan Chen$^{2}$ \footnote{Corresponding author. E-mail: x.d.chen@live.cn}
\\  \\\small $^{1}$ School of Mathematical Sciences, Xiamen University, \\\small Xiamen, 361005, P.R. China  \\\small $^{2}$ College of Mathematics and Information Science, Guangxi University, \\\small Nanning, 530004,  P.R. China }
\date{} \maketitle

\begin{abstract}
The weighted adjacency matrix $A_{f}(G)$ of a simple graph $G=(V,E)$ is the $|V|\times|V|$ matrix
whose $ij$-entry equals $f(d_{i},d_j)$,
where $f(x,y)$ is a symmetric function such that $f(d_i,d_j)>0$ if $ij\in E$ and  $f(d_i,d_j)=0$  if $ij\notin E$ and $d_i$ is the degree of the vertex $i$.
In this paper, we determine the unique graph having the largest spectral radius of $A_{f}(G)$ among all the bicyclic graphs under the assumption that $f(x,y)$ is increasing and convex in $x$ and $f(x_1,y_1)\geq f(x_2,y_2)$ when $|x_1-y_1|>|x_2-y_2|$ and $x_1+y_1=x_2+y_2$.
Moreover, we determine the unique graph having the second largest spectral radius of $A_{f}(G)$ among all the bicyclic graphs
when $f(x,y)=x+y$, $(x+y)^2$ or $x^2+y^2$, which corresponds to the well-known first Zagreb index, first hyper-Zagreb index, and forgotten index, respectively.
In addition, we also characterize the bicyclic graphs with the first two largest spectral radii of $A_{f}(G)$ when $f(x,y)=\frac{1}{2}(x/y+y/x)$,
corresponding to the extended index.

\end{abstract}

\noindent
{\bf Keywords:} bicyclic graphs, weighted adjacency matrix, extremal spectral radius\par
\noindent
{\bf Mathematics Subject Classification}: 05C50, 15A18.

 \baselineskip=0.30in

\section{Introduction}
 The topological indices play an important role in QSAR/QSPR  studies  in chemical graph theory. In order to uniformly study such kinds of  indices, it is natural and meaningful to find a uniform way to deal with as many indices as possible, see \cite{Ali, Pproperty, YaoYuedan, YJC}. Among these indices, the vertex-degree-based
indices received much attention \cite{Gutman2021,LiWang2021}, the general form of which could be represented by $\sum_{ij\in E} f(d_i,d_j)$, where $f$ is a symmetric real function, $E$ is the edge set of the graph and $d_i$ is the degree of the vertex $i$.

 On the other hand, as pointed out in \cite{LiTalk2015}, a matrix weighted by a function could keep more structural information than a single index did. To this end, Das et al. \cite{WeightedMatrix},  Li and Wang \cite{LiWang2021} proposed the notion of the {\em weighted adjacency matrix}:
\begin{definition}
For a graph $G=(V,E)$, the  {\it weighted adjacency matrix}  $A_{f}(G)=[a_{ij}]$ of a graph $G$ is defined as:
 $$
a_{ij}=\left\{\begin{matrix}
f(d_i,d_j), &\text{if } ij\in E(G),\\
0,                                                            &\text{otherwise},
\end{matrix}\right.
$$
where $f(x,y)$ is  a symmetric bivariate real function.
\end{definition}

In fact, lots of partial results have been obtained in earlier articles for some special functions $f(x,y)$, corresponding to some specified indices.  For example,  Chen \cite{Chen2} characterized the extremal
trees  with the largest and smallest spectral radii of ABC matrix in  the class of trees with $n$ vertices. Li and Wang \cite{LiABCUni} determined the extremal
unicyclic graphs on $n$ vertices with the largest and smallest spectral radii of ABC matrix. Yuan and Du \cite{YuanBicyclic} characterized the extremal bicyclic graphs on $n$ vertices with the first two maximum spectral radii of ABC matrix. And  the extended adjacency matrix  \cite{Ext1, Ext2, Ext3}, Zagreb matrix \cite{Gutman1972}, Randi\'c matrix \cite{RandicMatrix}, AG matrix \cite{Ext2, AGMatrix1, AGMatrix2} and Harmonic matrix \cite{HarmonicMatrix} are also studied one by one.

However, it is not easy to determine the spectral radii of $A_f(G)$ for general  symmetric bivariate real function $f(x,y)$ \cite{LiEnergy1, LiEnergy2, LiEnergy3, LiWang2021}. For this reason,  Li and Wang \cite{LiWang2021} added some particular restrictions on $f(x,y)$ and obtained the following results for trees.
\begin{theorem}\label{Lith1}\cite{LiWang2021} Assume that $f(x,y)>0$ is a symmetric real function, increasing and convex in variable $x$. Then the tree on $n$ vertices with the largest spectral radius of $A_f(T)$ is a star or a double star.
\end{theorem}

 \begin{theorem}\label{Lith2}\cite{LiWang2021}Assume that $f(x,y)$  has a form $P(x,y)$ or $\sqrt{P(x,y)}$ , where $P(x,y)$ is a symmetric polynomial with nonnegative coefficients and zero constant term. Then the tree on $n$ $(n\geq 9)$ vertices with the smallest spectral radius of
 $A_f(T)$ is uniquely a path.
 \end{theorem}
 For general graphs, Hu et al. \cite{Pproperty} added a stronger  restrictions on $f(x,y)$, namely the property $P$: if  $f(x_1,y_1)> f(x_2,y_2)$ when $|x_1-y_1|>|x_2-y_2|$ and $x_1+y_1=x_2+y_2$ in \cite{Pproperty}. Further, Zheng et al. \cite{ZhengRuiling} added another restriction on the symmetric bivariate real weighted function $f(x,y)$, namely the  property $P^*$: if\\
 $(i)$ $f(x,y)$ is increasing in  $x$;\\
 $(ii)$ $f(x,y)$ is convex in $x$; and\\
 $(iii)$ $f(x_1,y_1)\geq f(x_2,y_2)$ when $|x_1-y_1|>|x_2-y_2|$ and $x_1+y_1=x_2+y_2$.
\\And $A_{f}(G)$ is called the weighted adjacency matrix with property $P^*$ of $G$.

One can see that the function $f(x,y)$ with  property $P^*$ covers many indices, including the first Zagreb index  \cite{Gutman1972}: $f(x,y)=x+y$ , first hyper-Zagreb index \cite{hyperZagreb}: $f(x,y)=(x+y)^{2}$ , forgotten index \cite{ForgottenIndex}: $f(x,y)=x^2+y^2$ , general sum-connectivity index \cite{Zhou1}: $f(x,y)=(x+y)^{\alpha}$ when $\alpha\geq 1$, general Platt index \cite{Ali}: $f(x,y)=(x+y-2)^{\alpha}$ when $\alpha\geq 1$, general Sombor index \cite{GeneralSombor}: $f(x,y)= (x^{\alpha}+y^{\alpha})^{\beta}$ when $\alpha\geq 1$ and $\beta\geq 1$, exponential first Zagreb index \cite{ExpIndices}: $f(x,y)=e^{x+y}$, exponential general  sum-connectivity index \cite{ExpIndices}: $f(x,y)=e^{(x+y)^{\alpha}}$ when $\alpha\geq 1$  and even exponential general Sombor index \cite{ExpIndices}: $f(x,y)=e^{(x^{\alpha}+y^{\alpha})^{\beta}}$ when $\alpha\geq 1$ and $\beta\geq 1$.

For a graph $G$ with $n$ vertices and $m$ edges, we also call $G$ a $c$-{\it cyclic graph} where $c=m-n+1$. In particular, if $c=0$, $1$ or $2$, then $G$ is known as a tree, a unicyclic graph or a bicyclic graph, respectively. In \cite{ZhengRuiling}, Zheng et al.
determined the extremal trees with the smallest   and largest  spectral radii, and the extremal unicyclic graphs with the smallest  and  first three largest spectral radii of $A_{f}(G)$.

In this paper,  we consider bicyclic graphs. The paper is organized as follows. In the second section, we determine the unique graph with the largest spectral radius  of  $A_f(G)$  for the bicyclic graphs when $f(x,y)$ has the property $P^*$. In the third section, we characterize the  bicyclic graphs with the second largest spectral radius of  $A_f(G)$ weighted for some known indices.  In the final section,  we apply the Ruler's theorem of spectral radius  to characterize the bicyclic graphs with the first two  largest spectral radius of $A_f(G)$ when $f(x,y)=\frac{1}{2}(x/y+y/x)$.
\section{The bicyclic graphs with the largest $\rho(A_f(G))$ when $f(x,y)$ has the property $P^*$}

Firstly, we introduce some notations, terminologies and useful lemmas as follows.
Throughout the paper, $G=(V,E)$ denote a connected undirected simple
graph with $V=\{1,2,\ldots,n\}$ and
$|E|=m$.  We denote by $N(i)$ the neighbor set of vertex $i$, and $N[i]=N(i)\cup\{i\}$. As usual, we denote by  $d_i$ the
degree of $i$, i.e., $d_i=|N(i)|$. In particular, $\Delta(G)$ and $\delta(G)$, or $\Delta$ and $\delta$ for short, denote the maximum degree and minimum degree of $G$, respectively.
 The vertex $i$ is called a pendant vertex and an isolated vertex when  $d_i=1$ and $d_i=0$, respectively.
As usual, the {\it adjacency matrix} $A(G)=[a_{ij}]$ is the $n\times n$ matrix with $a_{ij}=1$ if $ij\in E$ and $a_{ij}=0$ otherwise.

 We denote by $\rho(A_{f}(G))$, or $\rho_f(G)$ for short, the spectral radius of $A_{f}(G)$, that is, the maximum absolute value of the eigenvalue of $A_{f}(G)$. In particular, we denote by $\rho(G)$ the spectral radius of the adjacency matrix of $G$.  Besides, we also denote by $\rho(M)$ the maximum eigenvalue of a square matrix $M$.
   If $G$ is connected and $f(d_i,d_j)>0$ for all vertices $i,j\in V(G)$, then by the Perron-Frobenius Theorem of non-negative irreducible matrices, $\rho_f(G)$
has multiplicity one and  $A_{f}(G)$ has a unique positive unit
eigenvector, which will be referred  as the {\it Perron vector} of $\rho_f(G)$ in the sequel.

Let $\mathcal{B}_{n}$ denote the class of all bicyclic graphs with $n$ vertices and $G\in\mathcal{B}_{n}$. The {\it base} of $G$, denoted by $\hat{G}$, is
the (unique) minimal connected bicyclic subgraph of $G$. It is easy
to see that $\hat{G}$ is the unique bicyclic subgraph of $G$
containing no pendant vertices, while $G$ can be  obtained from
$\hat{G}$ by attaching trees to some vertices of $\hat{G}$. If a tree is attached to the vertex $v$ of $\hat{G}$, then we denote the tree by $T(v)$. And we denote the star graph with $n$ vertices by $S_n$.

Let $C_{p}$  and $C_{q}$ be two vertex-disjoint cycles. Suppose that
$v\in V(C_{p})$ and  $u\in V(C_{q})$. In \cite{GuoShuguang}, Guo introduced
the $\infty$-graph  {\it $B(p,l,q)$} (Figure 1), which is  arisen from $C_{p}$
and $C_{q}$ by joining $v$ and $u$ by a path $(v=)w_{1}w_{2}\cdots
w_{l}(=u)$ of length $l-1$. Specially, $v=u$ when $l=1$ and  the path is the edge $vu$ when $l=2$.\par
Let $P_{p+1}$, $P_{l+1}$ and $P_{q+1}$ be three vertex-disjoint
paths, where $p, l, q\geq1$, $l=min\{p,l,q\}$ and at most one of them is $1$.
Identifying the three initial vertices and terminal vertices of
them, respectively, the resulting graph (Figure 1),  denoted by {\it
$P(p,l,q)$}, is called a $\theta$-graph in  \cite{GuoShuguang}.

\begin{figure}[htb]
\centering
\vspace*{1cm}\setlength{\belowcaptionskip}{0.1cm}
\includegraphics[width=0.65\textwidth]{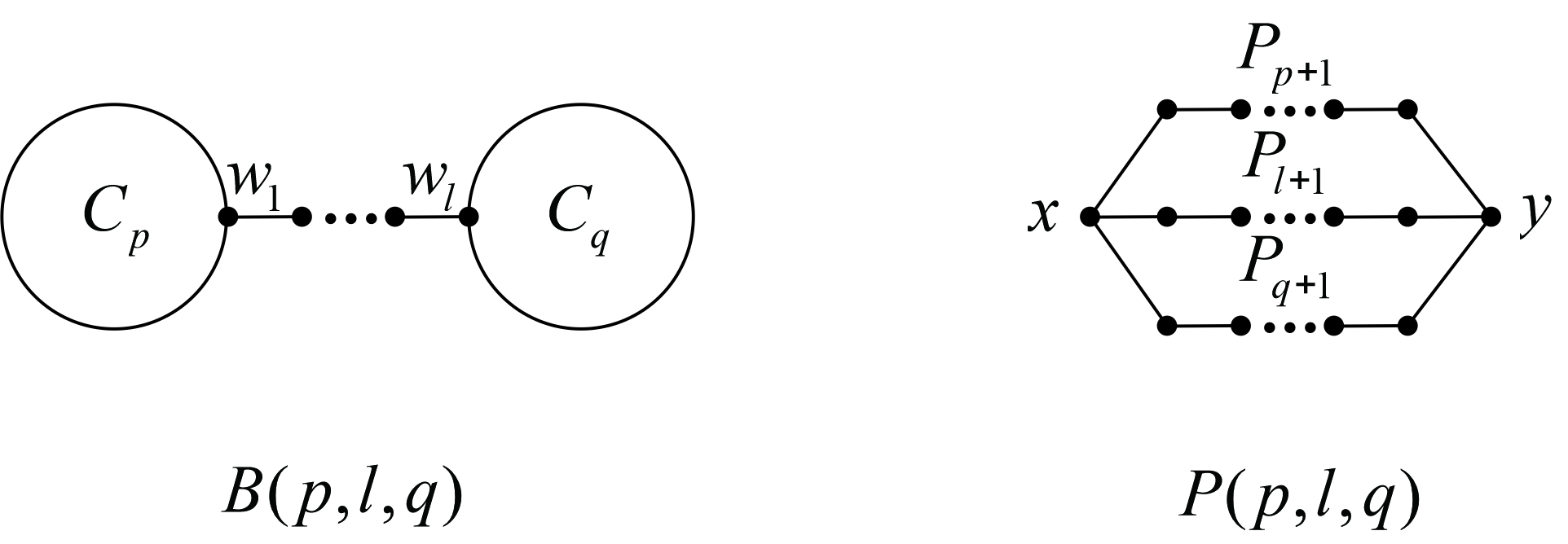}
\caption{The graphs $B(p,l,q)$ and $P(p,l,q)$.}\label{1f}
\end{figure}

 Now we define
the following two kinds of bicyclic graphs with  $n$ vertices :\hspace*{10pt} $\mathcal{B}^{*}(n)=\{G\in
\mathcal{B}(n)|$ $\hat{G}=B(p,l,q)\},$ \hspace*{10pt}
$\mathcal{B}^{**}(n)=\{G\in \mathcal{B}(n)|$ $\hat{G}=P(p,l,q)\}.$
Clearly, we have $\mathcal{B}(n)=\mathcal{B}^{*}(n) \cup\mathcal{B}^{**}(n)$.

 Let $G_1$ be the graph on $n$ vertices
obtained from $P(2,1,2)$ by joining  $n-4$ isolated vertices to the  vertex of degree $3$, $G_2$ be the graph on $n$ vertices
obtained from $B(3,1,3)$ by joining $n-5$ isolated vertices to the  vertex of degree $4$. Graphs $G_1$ and $G_2$ are shown in Figure \ref{fig2}.

\begin{figure}[htb]
\centering
\vspace*{1cm}\setlength{\belowcaptionskip}{0.1cm}
\includegraphics[width=0.45\textwidth]{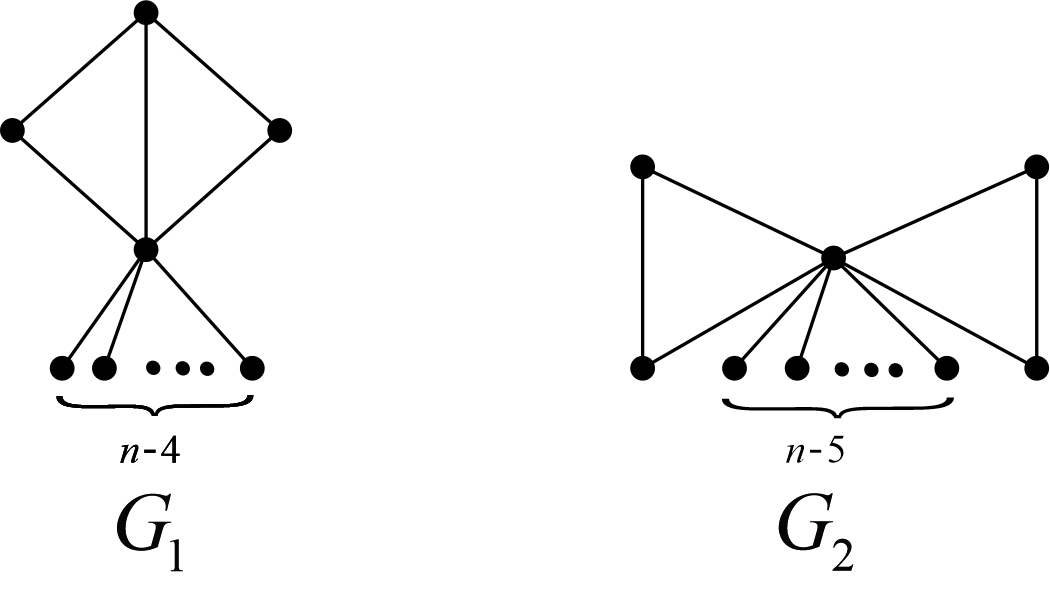}
\caption{The bicyclic graphs $G_1$ and $G_2$ .} \label{fig2}
\end{figure}

 We will show in the following that $G_2$ and $G_1$ are the graphs with maximal spectral radius of weighted adjacency matrices in $\mathcal{B}^{*}(n)$ and $\mathcal{B}^{**}(n)$, respectively. Furthermore, we show that $G_1$ is the unique graph with  the largest spectral radius  of weighted adjacency matrices in $\mathcal{B}(n)$.

\begin{definition} \label{d2.1} \cite{Kel1981}
Let $u$ and $v$ be two vertices of the graph $G$. The Kelmans operation from $u$ to $v$ is defined as follows: Delete the edge $uw$ and add a new edge $vw$ for all vertices $w\in N(u)-N[v]$.
\end{definition}

Notice that the two new graphs obtained from the same graph by Kelmans operation from $u$ to $v$ and from  $v$ to $u$ are isomorphic. So we usually call it "using the Kelmans operation on $u$ and $v$" in the sequel.

  \begin{lemma}\label{l2.1} \cite{BroSpe} Let $A$ be an $n\times n$ real symmetric matrix. Then the largest eigenvalue  $\lambda_1(A)= max_{||{\bf x}||=1} {\bf x}^{\top }A{\bf x}$, where ${\bf x} \in \mathbb{R}^{n}$.
\end{lemma}

 \begin{lemma} \cite{ZhengRuiling} \label{Kel}
 Let $G$ be a connected graph and $G'$ be the graph after a Kelmans operation on $G$.
 If $f(x,y)$ has the property $P^*$ and $G\not \cong G^{'}$, then $\rho(A_{f}(G^{'}))>\rho(A_{f}(G))$.
\end{lemma}

\begin{lemma} \cite{ZhengRuiling} \label{PendentVertex}
Let $v$ and $u$ be two vertices in graph $G$ such that the vertices in $N_1=N(v)-N[u]$ and $N_2=N(u)-N[v]$ are all pendent vertices and $1\leq |N_1|\leq |N_2|$. Let $w\in N_1$, $G'=G-vw+uw$. If $f(x,y)$ has the property $P^*$ and $G\not \cong G^{'}$,  then $\rho(A_{f}(G^{'}))>\rho(A_{f}(G))$.
\end{lemma}

\begin{lemma}\label{thB1} Let $G$ be a graph in $\mathcal{B}^{*}(n)$, where $n\geq 5$. If $f(x,y)$ has the property $P^*$, then $\rho(A_{f}(G))\leq \rho(A_{f}(G_2))$,
and the equality holds if and only if $G\cong G_2$.
\end{lemma}
\begin{proof} Choose $G\in \mathcal{B}^{*}(n)$ such that $\rho(A_{f}(G))$ is as large as possible. Suppose $\hat{G}=B(p,l,q)$ where $C_{p}=v_{1}v_{2}\cdots v_{p}v_{1}$, $C_{q}=u_{1}u_{2}\cdots u_{q}u_{1}$, and $w_{1}w_{2}\cdots w_{l-1}w_{l}$ is the unique path with length $l-1$ from $w_{1}=v_{1}\in V(C_{p})$ to  $w_{l}=u_{1}\in V(C_{q})$.

We first prove that $p=q=3$. Otherwise, without loss of generality, we suppose $p>3$.
Then we use Kelmans operation of $G$ on vertices $v_{1}$ and $v_{2}$ to obtain $G^{'}$. Note that $G^{'}\in \mathcal{B}^{*}(n)$, $\hat{G^{'}}=B(p-1,l,q)$ and  $\rho(A_{f}(G))<\rho(A_{f}(G^{'}))$ by Lemma \ref{Kel}, a contradiction. By repeating this process and we can conclude that $p=3$. Similarly, $q=3$.

Secondly, we prove that $d_{v_{2}}=d_{v_{3}}=d_{u_{2}}=d_{u_{3}}=2$ in $G$. By symmetry, we just need to prove $d_{v_{2}}=d_{v_{3}}=2$ in $G$. On the contrary, we suppose one tree $T$ is attached to $v_2$(resp. $v_3$), then we use Kelmans operation of $G$ on vertices $v_{1}$ and $v_{2}$(resp. $v_3$) to obtain $G^{'}$.
Note that $G^{'}\in \mathcal{B}^{*}(n)$, $\hat{G^{'}}=B(3,l,3)$, tree $T$ is not attached to $v_2$(resp. $v_3$) but $v_1$ in $G^{'}$, and  $\rho(A_{f}(G))<\rho(A_{f}(G^{'}))$ by Lemma \ref{Kel}, a contradiction. For convenience, denote the two trees attached to  $v_1$ and  $u_1$ by $T(v_1)$ and $T(u_1)$, respectively.

Moreover, we claim that $l=1$. Assume, on the contrary, that $l>1$. We use Kelmans operation of $G$ on vertices $w_{1}$ and $w_{2}$ to obtain $G^{'}$. Note that $G^{'}\in \mathcal{B}^{*}(n)$, $\hat{G^{'}}=B(3,l-1,3)$ if $l\geq 3$ and  $\rho(A_{f}(G))<\rho(A_{f}(G^{'}))$ by Lemma \ref{Kel}, a contradiction. And specially if $l=2$, after using Kelmans operation of $G$ on vertices $v_{1}$ and $u_{1}$ to obtain $G^{'}$,  the attached tree $T(u_1)$ becomes attaching to vertex $v_1$ in $G^{'}$. Therefore, there is at most one tree attaching to the vertex $v_{1}=u_{1}$ in the graph $G$ with maximal $\rho(A_{f}(G))$ and we denote this tree by $T^{'}=T^{'}(v_1)$ from now on. Now we know $G$ can be obtained from  the graph $B(3,1,3)$ by attaching  the tree $T^{'}$ to the vertex $v_1$.

Finally, we show that $T^{'}(v_1)\cong S_{n-4}$ and degree of $v_1$ in $T^{'}(v_1)$ is $n-5$. For convenience, we see the vertex $r=v_1$ as the root vertex of $T^{'}(v_1)$. On the contrary, there exists a vertex $r'\in N(r)\cap V(T^{'})$ such that $r'$ is not a pendent vertex, we use the Kelmans operation of $G$ on vertex $r'$ and other non-pendent vertices in $N(r)\cap V(T^{'})$ one by one until all vertices except $r'$ in $N(r)\cap V(T^{'})$ are pendent. Now we obtain a new graph $G^{'}\in \mathcal{B}^{*}(n)$ and the correponding new attaching tree $T^{''}=T^{''}(v_1)$. Replace $T^{'}$ by the smaller tree induced by $V(T^{''})-N[r]+{r'}$ and replace $r$ by $r'$, and repeat the process recursively. The whole process ends up with a graph obtained from $B(3,1,3)$ by attaching a caterpillar tree to $v_1$. Suppose the backbone of this caterpillar tree is path $(v_1=)x_1x_2\cdots x_t$, then we use Kelmans operation on vertices $x_i$ and $x_{i-1}$ from $i=t$ to $i=2$ and we would obtain the graph $G_2$. By Lemma \ref{Kel}, $\rho(A_{f}(G))< \rho(A_{f}(G_2))$, a contradiction.

Combining above arguments, we have $G\cong G_2$. This completes the proof.
\end{proof}

\begin{lemma}\label{thB2} Let $G$ be a graph in $\mathcal{B}^{**}(n)$, where $n\geq 4$. If $f(x,y)$ has the property $P^*$, then $\rho(A_{f}(G))\leq \rho(A_{f}(G_1))$,
and the equality holds if and only if $G\cong G_1$.
\end{lemma}
\begin{proof} Choose $G\in \mathcal{B}^{**}(n)$ such that $\rho(A_{f}(G))$ is as large as possible. Suppose $\hat{G}=P(p,l,q)$ and $P_{p+1}=xv_{1}v_{2}\cdots v_{p-1}y$, $P_{l+1}=xu_{1}u_{2}\cdots u_{l-1}y$, and $P_{q+1}=xw_{1}w_{2}\cdots w_{q-1}y$, $p, q\geq2$, $l=min\{p,l,q\}\geq1$. Specially, if $l=1$ then $P_{l+1}=xy$.

We first prove that $p=q=2$ and $l=1$. Assume, on the contrary, that $l>1$ or $p>2$ or $q>2$.
If $l>1$, then we use Kelmans operation of $G$ on vertices $x$ and $u_{1}$ to obtain $G^{'}$. Note that $G^{'}\in \mathcal{B}^{**}(n)$, $\hat{G^{'}}=P(p,l-1,q)$ and  $\rho(A_{f}(G))<\rho(A_{f}(G^{'}))$ by Lemma \ref{Kel}, a contradiction. Hence $l=1$. Similarly, if $p>2$, then we use Kelmans operation of $G$ on vertices $x$ and $v_{1}$ to obtain $G^{'}$. Note that $G^{'}\in \mathcal{B}^{**}(n)$, $\hat{G^{'}}=P(p-1,1,q)$ and  $\rho(A_{f}(G))<\rho(A_{f}(G^{'}))$ by Lemma \ref{Kel}, a contradiction. Hence $p=2$. By symmetry, we have $q=p=2$ and so $\hat{G^{'}}=P(2,1,2)$.

Secondly, we prove that at most one tree is attached to one vertex of degree $3$ in $P(2,1,2)$. On the contrary, we suppose one tree $T$ is attached to  vertex $v_1$(resp. $w_1$) of degree $2$ in $P(2,1,2)$, then we use Kelmans operation of $G$ on vertices $x$ and $v_{1}$(resp. $w_1$) to obtain $G^{'}$.
Note that $G^{'}\in \mathcal{B}^{**}(n)$, $\hat{G^{'}}=P(2,1,2)$, tree $T$ is not attached to $v_1$(resp. $w_1$) but $x$, and  $\rho(A_{f}(G))<\rho(A_{f}(G^{'}))$ by Lemma \ref{Kel}, a contradiction. Furthermore, suppose a tree has attached to $x$, if there exists another tree $T^{'}$ attaching to another vertex $y$ of degree $3$ in $P(2,1,2)$, then we use Kelmans operation of $G^{'}$
on vertices $x$ and $y$ to obtain $G^{''}$. Note that $G^{''}\in \mathcal{B}^{**}(n)$, tree $T^{'}$ is not attached to $y$ but $x$, and  $\rho(A_{f}(G^{'}))<\rho(A_{f}(G^{''}))$ by Lemma \ref{Kel}, a contradiction. For convenience, we denote by $T(x)$ the only attaching tree attached to  $x$. Now we know $G$ can be obtained from graph $P(2,1,2)$ by attaching  the tree $T(x)$ to the vertex $x$.

Finally, analogously with the proof of Lemma \ref{thB1}, we know that $T(x)\cong S_{n-3}$ and the degree of $x$ in $T(x)$ is $n-4$.

Combining above arguments, we have $G\cong G_1$. This completes the proof.
\end{proof}

\begin{theorem}\label{thB} Let $G$ be a graph in $\mathcal{B}(n)$, where $n\geq 4$. If $f(x,y)$ has the property $P^*$, then $\rho(A_{f}(G))\leq \rho(A_{f}(G_1))$,
and the equality holds if and only if $G\cong G_1$.
\end{theorem}
\begin{proof} When $n= 4$, the theorem holds obviously by Lemma \ref{thB2} since
$\mathcal{B}^{*}(n)=\emptyset$. When $n>4$, we show that $\rho(A_{f}(G_2))< \rho(A_{f}(G_1))$ as follows. We choose two non-adjacent vertices both with degree two in $G_2$ and use Kelmans operation of $G_2$ on the two vertices, then we obtain $G_1$. Hence $\rho(A_{f}(G_2))< \rho(A_{f}(G_1))$ by Lemma \ref{Kel}. Therefore, the theorem holds by Lemma \ref{thB1} and Lemma \ref{thB2}.
\end{proof}

\section{The bicyclic graphs with the second largest $\rho(A_f(G))$ when $f(x,y)=x+y, (x+y)^{2}$ or $x^{2}+y^{2}$}

In this section we would like to determine the bicyclic graphs with second largest $\rho(A_f(G))$, that is

\begin{figure}[htb]
\centering
\vspace*{1cm}\setlength{\belowcaptionskip}{0.1cm}
\includegraphics[width=0.47\textwidth]{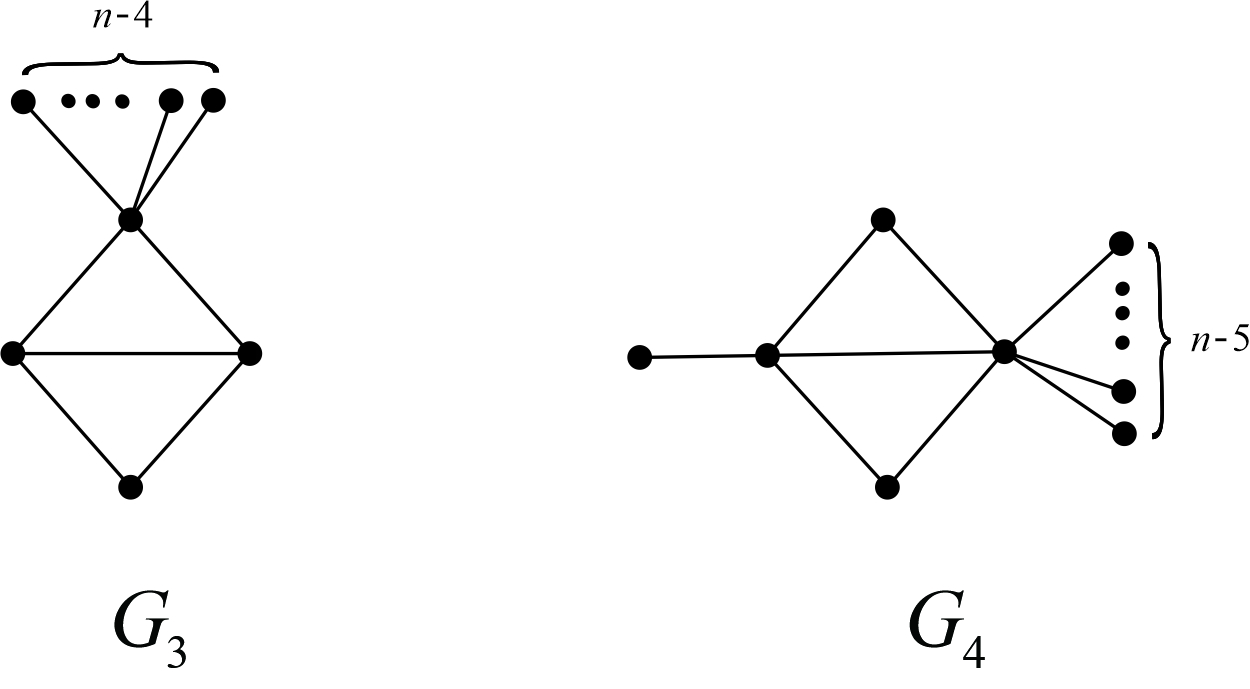}
\caption{The bicyclic graphs $G_3$ and $G_4$.} \label{G3-4}
\end{figure}

\begin{theorem}\label{secondB} Assume that $f(x,y)$ has the property $P^*$. Then the bicyclic graph on $n$ ($n\geq 6$) vertices with the second largest  $\rho(A_{f}(G))$ is $G_2$ or $G_3$ or $G_4$.
\end{theorem}
\begin{proof}
{\bf Case $1$. $G\in \mathcal{B}^{*}(n)$ ($n\geq 6$).}
According to Lemma \ref{thB1}, if $G\in \mathcal{B}^{*}(n)$, $f(x,y)$ has the property $P^*$, then $\rho(A_{f}(G))\leq \rho(A_{f}(G_2))$,
and the equality holds if and only if $G\cong G_2$, which is depicted in Figure \ref{fig2}.\\
{\bf Case $2$. $G\in \mathcal{B}^{**}(n)\setminus\{G_1\}$ ($n\geq 6$).}
 Choose $G\in \mathcal{B}^{**}(n)\setminus\{G_1\}$ such that $\rho(A_{f}(G))$ is as large as possible. Suppose $\hat{G}=P(p,l,q)$ and $P_{p+1}=xv_{1}v_{2}\cdots v_{p-1}y$, $P_{l+1}=xu_{1}u_{2}\cdots u_{l-1}y$, and $P_{q+1}=xw_{1}w_{2}\cdots w_{q-1}y$, $p, q\geq2$, $l=min\{p,l,q\}\geq1$. Specially, if $l=1$ then $P_{l+1}=xy$.\\
{\bf Case $2.1$.} If there exists more than two trees attaching on $\hat{G}=P(p,l,q)$, then obviously we can always find a new graph $G'\in \mathcal{B}^{**}(n)$ such that there exists exactly two trees attaching on $\hat{G'}$ and  $\rho(A_{f}(G))< \rho(A_{f}(G'))$ by repeating using Kelmans operation on $x$ and one of its adjacent vertex in $\hat{G}$, a contradiction.\\
{\bf Case $2.2$.} If there exists exactly two tree attaching on $\hat{G}=P(p,l,q)$, then similarly with the proof in Lemma \ref{thB2}, by repeating using Kelmans operation on $x$ (or $y$) and one of its adjacent vertex in $\hat{G}$, the two attaching trees must be $T(x)$ and $T(y)$ such that $x$ and $y$ are attaching vertices  and $p=q=2$, $l=1$. Moreover, by the maximality of $\rho(A_f(G))$, analogously with the proof in Lemma \ref{thB1}, $T(x)$ and $T(y)$ must be star graphs. Furthermore, the two star graphs must be $S_2$ and $S_{n-4}$ by Lemma \ref{PendentVertex}. So the graph $G\cong G_4$, which is depicted in Figure \ref{G3-4}.\\
{\bf Case $2.3$.} If there exists exactly one tree attaching on $\hat{G}=P(p,l,q)$, then by the maximality of $\rho(A_{f}(G))$, it is easy to see that $\hat{G}=P(2,1,2)$ by repeating using Kelmans operation on $x$ (or $y$) and one of its adjacent vertex in $\hat{G}$, in which the proof is similar with Lemma \ref{thB2}. Furthermore, since $G \not \cong G_1$, then $G\cong G_3$, which is depicted in Figure \ref{G3-4}.

Combining above arguments, we can conclude that the bicyclic graphs on $n$ ($n\geq 6$) vertices with the second largest spectral radius of $A_{f}(G)$ is $G_2$ or $G_3$ or $G_4$.
\end{proof}

By computations, the approximate values of $\rho(A_f(G_i))$ ($i=2,3,4$) when $f(x,y)=1, (x+y), (x+y)^2, (x+y)^3$ for $n=6$ and $n=7$ are shown in two tables in the Appendix. From the two tables, we discover that the bicyclic graphs with the second largest  $\rho(A_f(G))$ depend on the function $f(x,y)$ and the order $n$  when $n$ is small, i.e. the extremal bicyclic graph with the second largest $\rho(A_f(G))$ is not unique for general function $f(x,y)$  with the property $P^*$ and the order $n$. However, we also find that for some well-konwn indices, such as  first Zagreb index $f(x,y)=x+y$, first hyper-Zagreb index $f(x,y)=(x+y)^{2}$ and forgotten index $f(x,y)=x^2+y^2$, the bicyclic graphs with the second largest  $\rho(A_f(G))$ not depend on the order $n$ when $n$ is a little bit larger. So next we would like to show that the bicyclic graphs with the second largest  $\rho(A_f(G))$ for the above three indices are all the graph $G_2$ when $n$ is a little bit large, that is

\begin{theorem} \label{Zagreb} The bicyclic graphs on $n$ vertices with the second largest $\rho(A_f(G))$ are all the graph $G_2$, if one of the following conditions holds.\\
$(i)$ The function is first Zagreb index $f(x,y)=x+y$ and $n\geq 10$.\\
$(ii)$ The  function is first hyper-Zagreb index $f(x,y)=(x+y)^{2}$ and $n\geq 9$.\\
$(iii)$ The  function is forgotten index $f(x,y)=x^2+y^2$ and $n\geq 8$.
\end{theorem}
\begin{proof}
To prove  Theorem \ref{Zagreb}, firstly we shall introduce some important results about matrix from \cite{BroSpe}.

\begin{lemma}\label{symmetric}
 Let $M$ be a real symmetric matrix. Then the eigenvalues of $M$ are all real numbers and their sum is euqal to the sum of the diagonal entries of $M$.
\end{lemma}

\begin{lemma}\label{PFtheorem} An irreducible, nonnegative $m\times m$ matrix $M$ always has a real, positive eigenvalue $\rho_1$, so that:\\
$(i)$ $|\rho_i|\leq \rho_1$ holds for all other (possibly complex) eigenvalues
$\rho_i$, $i=2,\ldots,m$,\\
$(ii)$ $\rho_1$ is a simple zero of the characteristic polynomial of $M$, and\\
$(iii)$ the eigenvector ${\bf x_1}$ corresponding to $\rho_1$ has positive components.
\end{lemma}

Suppose that $M$ is a real symmetric matrix whose rows and columns are indexed by
$R=\{1,2,\ldots,m\}$. Let ${R_1,R_2,\ldots,R_n}$ be a partition of $R$. Let $M$ be partitioned according to ${R_1,R_2,\ldots,R_n}$, that is,
\begin{equation*}
 M=   \begin{bmatrix}
    M_{1,1} & \ldots & M_{1,n}\\
    \vdots & \ddots & \vdots\\
    M_{n,1} & \ldots & M_{n,n}
    \end{bmatrix}
\end{equation*}
where $M_{i,j}$ denotes the block of $M$ formed by rows in $R_i$ and the columns in $R_j$. Let $b_{i,j}$ denote the average row sum of $M_{i,j}$. Then the matrix $N=[b_{i,j}]$ is called the {\em quotient matrix}. If the row sum of each block $M_{i,j}$ is constant, then the partition is called {\em equitable}.

\begin{lemma}\label{equitable}
Let $M\geq 0$ be an irreducible real symmetric matrix, $N$ be the quotient matrix of an equitable partition of $M$. If $\lambda$ is an eigenvalue of $N$, then $\lambda$ is also an eigenvalue of  $M$. Furthermore, $\rho(N)=\rho(M)$.
 \end{lemma}
What's more, we introduce the {\bf Descartes' Rule of Signs} in \cite{DescartesRule}.
\begin{lemma} \cite{DescartesRule} \label{DescartesRule}
Let $p(x)=a_0x^{b_0}+a_1x^{b_1}+\cdots+a_nx^{b_n}$ denote a polynomial with nonzero real coefficients $a_i$, where the $b_i$ are integers satisfying $0\leq b_0 < b_1 <\cdots< b_n$. Then the number of positive real zeros of $p(x)$ is either equal to the number of variations in sign in the sequence $a_0,\ldots,a_n$ of the coefficients or less than that by an even number. The number of negative real zeros of $p(x)$ is either equal to the number of variations in sign in the sequence of the coefficients of $p(-x)$ or less than that by an even number.
\end{lemma}

By the symmetry of the graph $G_2$ on $n\geq 6$ vertices, let $R_1=\{v|d_v=n-1\}$, $R_2=\{v|d_v=2\}$ and $R_3=\{v|d_v=1\}$. Then the quotient matrix of the equitable partition $\{R_1, R_2, R_3\}$ is
\begin{equation*}
 A_1=   \begin{bmatrix}
    0 & 4f(n-1,2) & (n-5)f(n-1,1)\\
    f(n-1,2) & f(2,2)& 0\\
    f(n-1,1) & 0& 0
    \end{bmatrix}.
\end{equation*}

The characteristic polynomial of $A_1$ is
$$\phi_1(\lambda)=
\lambda^{3}-f(2,2)\lambda^{2}-[4f^{2}(n-1,2)+(n-5)f^{2}(n-1,1)]\lambda+(n-5)f^{2}(n-1,1) f(2,2),$$
where $n\geq 6$.

Similarly, when $n\geq 6$, we can obtain one quotient matrix of the equitable partition of $A_f(G_4)$ is
\begin{equation*}
 A_2=   \begin{bmatrix}
    0 & 2f(n-2,2) & f(n-2,4) & 0 & (n-5)f(n-2,1)\\
    f(n-2,2) & 0 & f(4,2) & 0 & 0\\
    f(n-2,4) & 2f(4,2) & 0 & f(4,1) & 0\\
    0 & 0 & f(4,1) & 0 & 0\\
    f(n-2,1) & 0 & 0 & 0 & 0
    \end{bmatrix}.
\end{equation*}
The characteristic polynomial of $A_2$ is
\begin{eqnarray}
\phi^{'}_2(\lambda)&=&\lambda^{5}-[2x_1^{2}+x_2^{2}+(n-5)x_3^{2}+2x_4^{2}+x_5^{2}]\lambda^{3} -4x_1x_2x_4\lambda^{2}\nonumber\\
&&+[2x_1^{2}x_5^{2}+(2n-10)x_3^{2}x_4^{2}+(n-5)x_3^{2}x_5^{2}]\lambda,\nonumber
\end{eqnarray}
where $n\geq 6$, $x_1=f(n-2,2)$, $x_2=f(n-2,4)$, $x_3=f(n-2,1)$, $x_4=f(4,2)$ and  $x_5=f(4,1)$.
For convenience, let
\begin{eqnarray}
\phi_2(\lambda)&=&\lambda^{4}-[2x_1^{2}+x_2^{2}+(n-5)x_3^{2}+2x_4^{2}+x_5^{2}]\lambda^{2} -4x_1x_2x_4\lambda\nonumber\\
&&+2x_1^{2}x_5^{2}+(2n-10)x_3^{2}x_4^{2}+(n-5)x_3^{2}x_5^{2}.\nonumber
\end{eqnarray}
Obviously, the maximum zeros of $\phi^{'}_2(\lambda)$ and $\phi_2(\lambda)$ are the same.

Moreover, when $n\geq 5$, we can also obtain one quotient matrix of the equitable partition of $A_f(G_3)$ is
\begin{equation*}
 A_3=   \begin{bmatrix}
    0 & 2f(n-2,3) & 0 & (n-4)f(n-2,1)\\
    f(n-2,3) & f(3,3) & f(3,2) & 0\\
    0 & 2f(3,2) & 0& 0\\
    f(n-2,1) & 0 & 0 & 0
    \end{bmatrix}.
\end{equation*}
The characteristic polynomial of $A_3$ is
\begin{eqnarray}
\phi_3(\lambda)&=&\lambda^{4}-y_3\lambda^{3}-[(n-4)y_2^{2}+2y_1^{2}+2y_4^{2}]\lambda^{2} +(n-4)y_2^{2}y_3\lambda\nonumber\\
&&+(2n-8)y_2^{2}y_4^{2},\nonumber
\end{eqnarray}
where  $n\geq 5$, $y_1=f(n-2,3)$, $y_2=f(n-2,1)$, $y_3=f(3,3)$ and $y_4=f(3,2)$.

For convenience, we suppose the maximum zeros of $\phi_1$, $\phi_2$ and $\phi_3$ are $\lambda_1$, $\lambda_2$ and $\lambda_3$, respectively.

\begin{lemma}\label{phi1}
If $f(x,y)$ has the property $P^*$, then $\rho(A_f(G_2))>\sqrt{n-1}f(n-1,1)$ when $n\geq 6$.
\end{lemma}
\begin{proof}
 By Lemma \ref{symmetric} and \ref{equitable}, we denote by $\lambda_1=r_1\geq r_2\geq r_3$ the three real zeros of $\phi_1(\lambda)$. From $\lambda_1=\rho(A_f(G_2))>0$, and  the coefficients of $\lambda$ and the constant term are positive when $n\geq 6$, then by Lemma \ref{DescartesRule}, it is easy to know  $r_1\geq r_2>0>r_3$. And notice that
 \begin{eqnarray}
\phi_1(\sqrt{n-1}f(n-1,1))&=&(n-1)\sqrt{n-1}f^3(n-1,1)-(n-1)f(2,2)f^2(n-1,1)\nonumber\\
&&-[4f^2(n-1,2)+(n-5)f^2(n-1,1)]\sqrt{n-1}f(n-1,1)\nonumber\\
&&+(n-5)f^2(n-1,1)f(2,2)\nonumber\\
&&\leq -4f(2,2)f^2(n-1,1)\nonumber\\
&&<0.\nonumber
\end{eqnarray}
 And since $\lim\limits_{\lambda\rightarrow +\infty}\phi_1(\lambda)=+\infty$, then we have $\lambda_1> \sqrt{n-1}f(n-1,1)$.
\end{proof}

%
\noindent{\bf The proof of Theorem \ref{Zagreb}.} 
 By Lemma \ref{symmetric} and Lemma \ref{equitable}, we know all zeros of $\phi_1$, $\phi_2$ and $\phi_3$ are real. For convenience, we denote by $\lambda_2=a_1\geq a_2\geq a_3\geq a_4$ the four zeros of $\phi_2(\lambda)$ and $\lambda_3=b_1\geq b_2\geq b_3\geq b_4$ the four zeros of $\phi_3(\lambda)$.  By Lemma \ref{PFtheorem} and Lemma \ref{equitable}, $\lambda_2=\rho(A_f(G_4))> 0$ and $\lambda_3=\rho(A_f(G_3))> 0$. Furthermore, the coefficients of $\lambda^4$ and the constant term  are positive while the coefficients of $\lambda^2$ and $\lambda$ are negative in $\phi_2$ when $n\geq 6$. Then by Lemma \ref{DescartesRule}, it is easy to know  $a_1\geq a_2>0>a_3\geq a_4$. Similarly, it is easy to know  $b_1\geq b_2>0>b_3\geq b_4$ when $n\geq 5$ by Lemma \ref{DescartesRule}.

For $f(x,y)=x+y$, by Lemma \ref{phi1}, the maximum zeros $\lambda_1$ of $\phi_1$ satifies: $\lambda_1> \sqrt{n-1}f(n-1,1)=n\sqrt{n-1}$ when $n \geq 6$.
$$\phi_2(n\sqrt{n-1})=1067n - 48n^2\sqrt{n-1} - 24n^3\sqrt{n-1} - 533n^2 + 16n^3 - 18n^4 + 3n^5 - 485 >0$$
when $n \geq 10$, and
 \begin{eqnarray}
\phi_2((n-5)\sqrt{n-1})&=&240n\sqrt{n-1}+72n^2\sqrt{n-1} - 24n^3\sqrt{n-1}- 3668n\nonumber\\
&& + 1577n^2 - 489n^3 + 97n^4 - 7n^5 +  2540\nonumber\\
&&<0\nonumber
\end{eqnarray}
when $n \geq 7$. Notice that $\phi_2(0)>0$. It follows that $n\sqrt{n-1}> \lambda_2> (n-5)\sqrt{n-1}$ when $n \geq 10$. Besides,
$$\phi_3(n\sqrt{n-1})=(n - 1)(2n^2 - 30n^2\sqrt{n-1} - 250n - 13n^3 + 3n^4 + 24n\sqrt{n-1} + 200) >0$$
when $n \geq 9$, and
 \begin{eqnarray}
\phi_3((n-5)\sqrt{n-1})&=&50(n - 4)(n - 1)^2 + (n - 1)^2(n - 5)^4 - 6(n - 1)^{3/2}(n - 5)^3\nonumber\\
&&- (n - 1)(n - 5)^2[2(n + 1)^2\nonumber+ (n - 1)^2(n - 4) + 50]\\
&&+ 6(n - 1)^{5/2}(n - 4)(n - 5) \nonumber\\
&&<0\nonumber
\end{eqnarray}
when $n \geq 8$. Notice that $\phi_3(0)>0$. It follows that $n\sqrt{n-1}> \lambda_3> (n-5)\sqrt{n-1}$ when $n \geq 9$.

Combining above arguments, $\lambda_1>n\sqrt{n-1}> max\{\lambda_2,\lambda_3\}$ for $f(x,y)=x+y$ when $n \geq 10$.

Similarly, for $f(x,y)=(x+y)^2$, we can obtain $\lambda_1>n^2\sqrt{n-1}$ when $n \geq 6$, $n^2\sqrt{n-1}>\lambda_2>(n-5)^2\sqrt{n-1}$ when $n \geq 9$,
and $n^2\sqrt{n-1}>\lambda_3>(n-5)^2\sqrt{n-1}$ when $n \geq 9$. Hence, $\lambda_1>n^2\sqrt{n-1}> max\{\lambda_2,\lambda_3\}$ for $f(x,y)=(x+y)^2$ when $n \geq 9$.

Analougously, for $f(x,y)=x^2+y^2$, we can obtain $\lambda_1>((n-1)^2+1)\sqrt{n-1}>(n-1)^2\sqrt{n-1}$ when $n \geq 6$, $(n-1)^2\sqrt{n-1}>\lambda_2>(n-5)^2\sqrt{n-1}$ when $n \geq 8$,
and $(n-1)^2\sqrt{n-1}>\lambda_3>(n-5)^2\sqrt{n-1}$ when $n \geq 8$. Hence, $\lambda_1>(n-1)^2\sqrt{n-1}> max\{\lambda_2,\lambda_3\}$ for $f(x,y)=x^2+y^2$  when $n \geq 8$.

Therefore, according to Theorem \ref{secondB}, the theorem holds.
\end{proof}

\section{The bicyclic graphs with the first two largest  $\rho(A_{f}(G))$ when $f(x,y)=\frac{1}{2}(x/y+y/x)$}
However, some weighted functions in  weighted adjacency matrices don't have the property $P^*$, so the method of using Kelmans operation in the previous two sections seems to fail.  For example, the  extended adjacency matrix $A_{ex}(G)=[a_{ij}]$ , introduced by Yang et al \cite{YangExt}, is defined as $a_{ij}=\frac{1}{2}\big(d_i/d_j + d_j/d_i\big)$ if $ij\in E$ and $a_{ij}=0$ otherwise.  Its weighted function is the extended index $f(x,y)=\frac{1}{2}(x/y+y/x)$ \cite{LiWang2021}, which has no property $P^*$.   Denote the spectral radius of $A_{ex}(G)$ by $\rho_{ex}(G)$ and recall that we denote the spectral radius of $A(G)$ by $\rho(G)$.
In this section, we will characterize the bicyclic graphs with the first two largest $\rho_{ex}(G)$ (i.e. Theorem \ref{th4.1}), in which the proof technique is using the Ruler's theorem of spectral radius.

\begin{theorem}\label{th4.1}
Among all the bicyclic graphs of order $n\geq12$,
 $G_1$ and $G_2$ are the bicyclic graphs with the first two largest $\rho_{ex}(G)$, respectively. Furthermore, $$\rho_{ex}(G_1)>\frac{1}{2}(n-0.9)\sqrt{n-3.8}>\rho_{ex}(G_2)>\frac{1}{2}(n-0.9)\sqrt{n-5},$$
 where $G_1$ and $G_2$ are depicted in Figure \ref{fig2}.
\end{theorem}

\begin{proof}
To prove Theorem \ref{th4.1}, firstly we shall introduce some notations, terminologies and useful lemmas as follows.

Let $\mathcal{S}(n, \Delta, c)$ be the class of $c$-cyclic graphs on
$n$ vertices with fixed maximum degree $\Delta$.
Let $H(n, \Delta, c)$, as shown in Figure \ref{fig3-4}, be the $c$-cyclic graph with $n$ vertices and maximum degree $\Delta$, where $c \geq 1$ and $\Delta \geq \frac{1}{2}(n+c+1) $. Let $\phi_{ex}(G;x)$ 
be the characteristic polynomial of the extended adjacency matrix of a graph $G$.


\begin{figure}[ht]
\centering
\vspace*{1cm}\setlength{\belowcaptionskip}{0.1cm}
\includegraphics[width=0.8\textwidth]{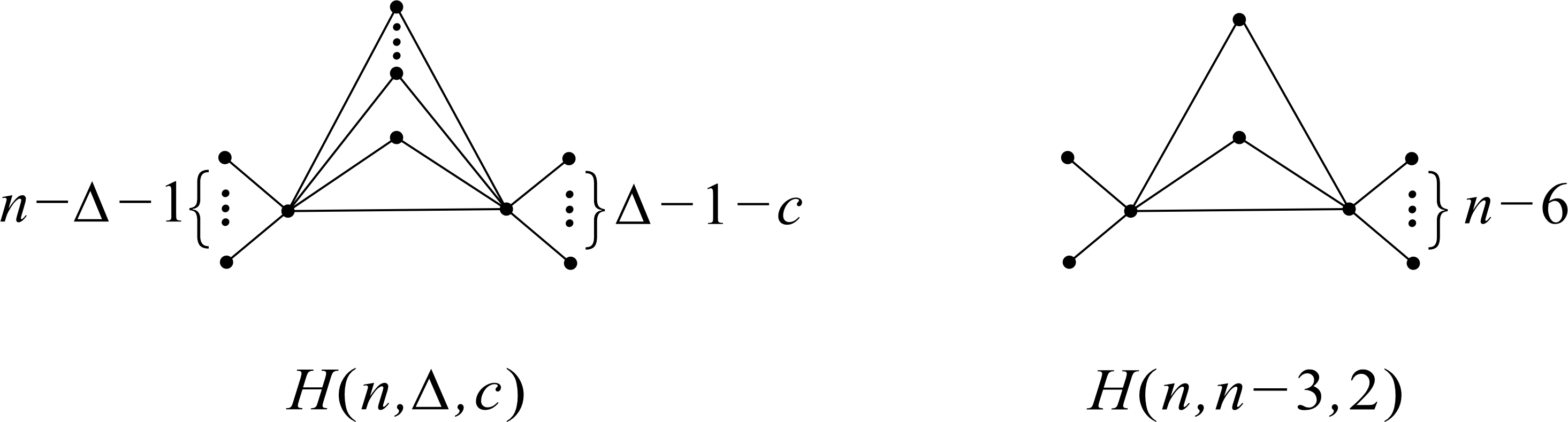}
\caption{The $c$-cyclic graphs $H(n,\Delta,c)$ and $H(n,n-3,2)$.} \label{fig3-4}
\end{figure}

\begin{lemma}\cite{Ext2}\label{l4.1}
For any graph $G$ (with $\delta(G)\geq 1$), we have
\begin{eqnarray*}
\rho(G)\leq \rho_{ex}(G)\leq\frac{1}{2}\bigg(\frac{\Delta(G)}{\delta(G)}+\frac{\delta(G)}{\Delta(G)}\bigg)\rho(G).
\end{eqnarray*}
The left equality holds if and only if $G$ is a regular graph and the right equality holds if and only if $G$ is a regular graph or a  semiregular bipartite graph.
\footnote{The original statement on the right equality is that it holds if and only if $G$ is a complete bipartite graph, but it's not quite right.}
\end{lemma}

Next we would like to introduce the {\bf Ruler's theorem} of spectral radius as follows.
\begin{lemma}\cite{Liuc-cyclic}\label{l4.2}
Let $G$ be the graph with the largest spectral radius in $\mathcal{S}(n,\Delta,c)$, where $c\geq1$ and $\Delta \geq \frac{1}{2}(n+c+1)$. If $n\geq6$ and $c=2$ (i.e. $G$ is a bicyclic graph), then $G\cong H(n,\Delta,2)$.
\end{lemma}

\begin{lemma}\cite{Liuc-cyclic}\label{l4.3}
Let $G$ be the graph with the largest spectral radius in $\mathcal{S}(n,\Delta,c)$. If $\Delta\leq n-2$, then there must exist some graph $G^{'}\in \mathcal{S}(n,\Delta+1,c)$ such that $\rho(G)<\rho(G^{'})$.
\end{lemma}

\begin{lemma}\label{l4.4}
Let $H(n, n-3, 2)$ be a bicyclic graph shown as Figure \ref{fig3-4}.\\
\rm{(i)} $\rho(H(n, n-3, 2))<\sqrt{n}$ for $n\geq12$.\\
\rm{(ii)} $\rho(H(n, n-3, 2))<\sqrt{n-1.2}$ for $n\geq20$.
\end{lemma}
\begin{proof}
By elementary computation, we get $\phi_{ex}(H(n, n-3, 2);x)=x^{n-4}h_n(x)$, where
\begin{eqnarray*}
h_n(x)=x^4-(n+1)x^2-4x+4(n-5).
\end{eqnarray*}
Denote by $x_1\geq x_2\geq x_3\geq x_4$ the roots of $h_n(x)=0$.
Note that $x_4<0, x_1>0$ and $h_n(0)=4(n-5)>0$ for $n\geq12$. Thus $x_3<0<x_2$.
It is easy to check that
\begin{eqnarray*}
h_n(\sqrt{n-3})=4n + (n - 3)^2 - 4\sqrt{n-3} - (n + 1)(n - 3) - 20<0,
\end{eqnarray*}
\begin{eqnarray*}
h_n(\sqrt{n})=4n - n(n + 1) + n^2 - 4\sqrt{n} - 20>0
\end{eqnarray*}
for $n\geq12$,
\begin{eqnarray*}
h_n(\sqrt{n-1.2})=4n + (n - 1.2)^2 - 4\sqrt{n - 1.2} - (n + 1)(n - 1.2) - 20>0
\end{eqnarray*}
for $n\geq20$. Then we obtain the desired result.
\end{proof}


In the following, we partition all the bicyclic graphs of order $n$ into three parts, based on the maximum degree $\Delta$:\\
$(i) \Delta=n-1$ (Lemma \ref{l4.5});\\
$(ii) \Delta=n-2$ (Lemma \ref{l4.6});\\
$(iii) \Delta\leq n-3$ (Lemma \ref{l4.7}).

Firstly, let us consider the bicyclic graphs of order $n$ whose maximum degree is $n-1$. It is easy to see that there are two possibilities for such graphs: $G_1$ or $G_2$, which are depicted in Figure \ref{fig2}. The estimates about $\rho_{ex}(G_1)$ and $\rho_{ex}(G_2)$ are presented as following.\\
\begin{lemma}\label{l4.5}
Let  $G_1$ and $G_2$ be the bicyclic graphs of order $n$.
Then $\rho_{ex}(G_1)> \rho_{ex}(G_2)>\frac{1}{2}(n-0.9)\sqrt{n-5}$ for $n\geq12$.
\end{lemma}
\begin{proof}
By direct computation, $\phi_{ex}(G_1;x)=\frac{x^{n-4}}{288(n-1)^2}h_{n, 1}(x)$, where
\begin{eqnarray*}
&&h_{n, 1}(x)\\
&&=288(n-1)^2x^4-52(4+(n-1)^2)(9+(n-1)^2)x+169(n-4)(1+(n-1)^2)^2\\
&&-\big[36(4+(n-1)^2)^2+8(9+(n-1)^2)^2
+72(n-4)(1+(n-1)^2)^2+676(n-1)^2\big]x^2,
\end{eqnarray*}
On one hand, it is easy to check that
\begin{eqnarray*}
&&h_{n, 1}\Big(\frac{1}{2}(n-0.9)\sqrt{n-3.8}\Big)\\
&&~~~~~~~~~~=\frac{\sqrt{n-3.8}}{5}(-130n^5+637n^4-2938n^3+6123n^2-10010n+5850)+\\
&&~~~~~~~~~~~~~~\frac{1}{125000}(-475000n^7-1177500n^6+32155750n^5-133858525n^4+\\
&&~~~~~~~~~~~~~~282465650n^3-381185036n^2+338531472n-198949811)<0
 \end{eqnarray*}
for $n\geq12$, it then follows that $\rho_{ex}(G_1)>\frac{1}{2}(n-0.9)\sqrt{n-3.8}$ .

And we have $\phi_{ex}(G_2;x)=\frac{x^{n-6}(x-1)(x+1)^2}{4(n-1)^2}h_{n, 2}(x)$, where
\begin{eqnarray*}
h_{n, 2}(x)&=&4(n-1)^2x^3-4(n-1)^2x^2-[(4+(n-1)^2)^2+(n-5)(1+(n-1)^2)^2]x\\
&&+(n-5)(1+(n-1)^2)^2.
\end{eqnarray*}
Similarly, it is easy to check that  $h_{n, 2}\big(\frac{1}{2}(n-0.9)\sqrt{n-3.8}\big)>0$, $h_{n, 2}\big(\frac{1}{2}(n-0.9)\sqrt{n-5}\big)<0$ and $h_{n, 2}(x)>0$ when $x>\frac{1}{2}(n-0.9)\sqrt{n-3.8}$ for $n\geq12$. And hence,  $\frac{1}{2}(n-0.9)\sqrt{n-5}<\rho_{ex}(G_2)<\frac{1}{2}(n-0.9)\sqrt{n-3.8}$. This completes the proof of the lemma.
\end{proof}

The characterization of all the bicyclic graphs of order $n$ with maximum degree $\Delta =n-2$ is straightforward.\\
{\bf Claim $4.1$} Let $G$ be a bicyclic graph of order $n\geq12$ with maximum degree $n-2$. Then $G\cong D_{n,i}$ for some $1\leq i\leq 9$, where $D_{n,i}$, $1\leq i\leq 9$, are depicted in Figure \ref{fig5}.
\begin{figure}[htb]
\centering
\includegraphics[width=0.6\textwidth]{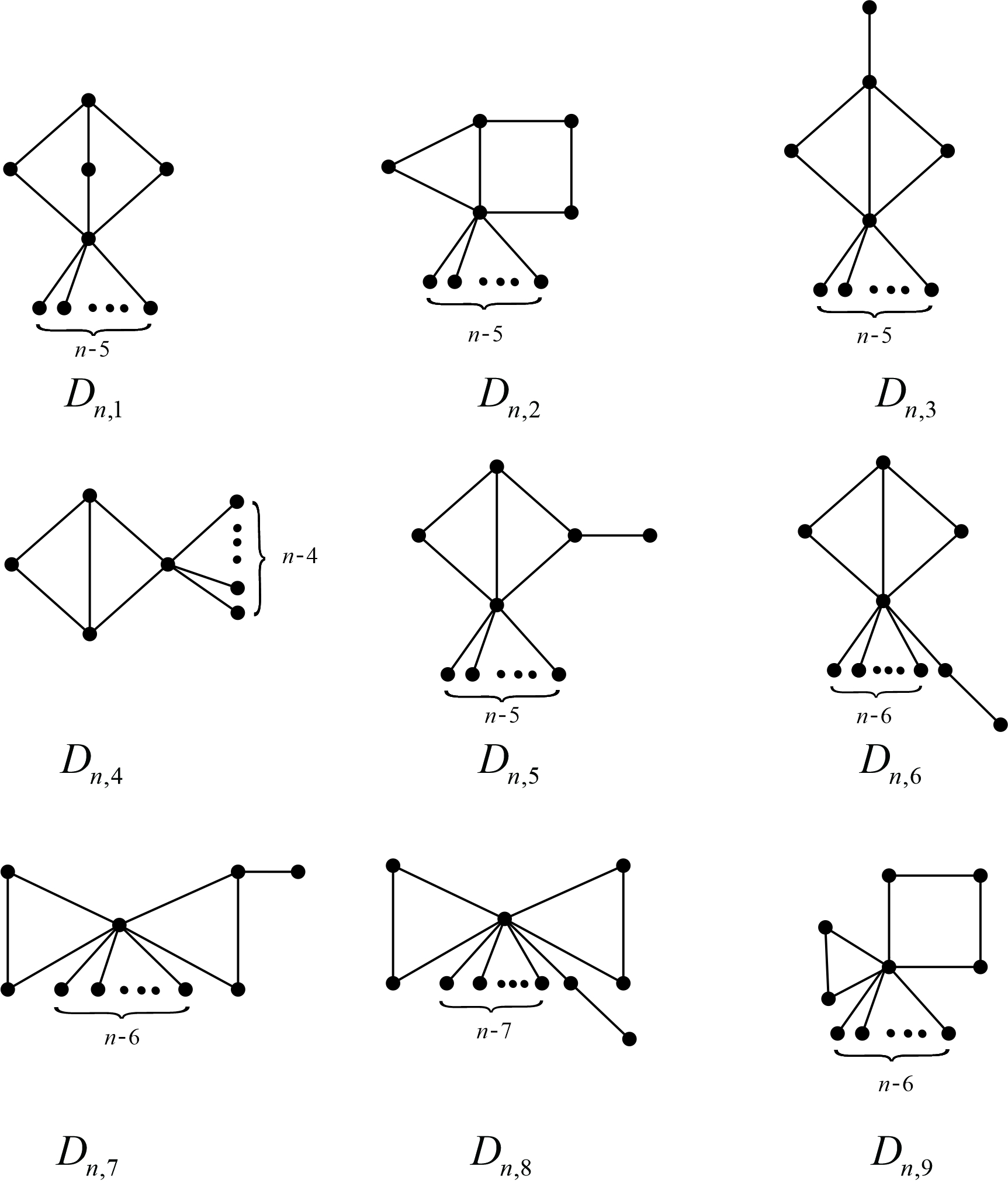}
\caption{The bicyclic graphs of order $n$ with $\Delta=n-2$.} \label{fig5}
\end{figure}

\begin{lemma} \label{l4.6}
Let $G$ be a bicyclic graph of order $n\geq12$ with maximum degree $n-2$. Then $\rho_{ex}(G)<\frac{1}{2}(n-0.9)\sqrt{n-5}$.
\end{lemma}
\begin{proof}
By Claim $4.1$, We can obtain that $G\cong D_{n,i}$ for $1\leq i\leq9$. To complete this proof, we just need  to show that $\rho_{ex}(D_{n,i})<\frac{1}{2}(n-0.9)\sqrt{n-5}$ for each $1\leq i\leq 9$ by Lemma \ref{l4.5}. We would take $D_{n,1}$ as an example to illustrate how to prove it, and the proofs of other $D_{n,i}$ with $2\leq i\leq9$ are similar. By direct computation, we have $\phi_{ex}(D_{n,1};x)=\frac{x^{n-4}}{2304(n-2)^2}h_{n,3}(x)$, where
\begin{eqnarray*}
h_{n,3}(x)&=&2304(n-2)^2x^4+2028(n-5)(1+(n-2)^2)^2\\
&&-\big(432(4+(n-2)^2)^2+576(n-5)(1+(n-2)^2)^2+8112(n-2)^2\big)x^2.
\end{eqnarray*}

On the other hand, for $n\geq12$, we get
\begin{eqnarray*}
&&h_{n,3}\Big(\frac{1}{2}(n-0.9)\sqrt{n-7}\Big)\\
&&~~~~~~~~~~=\frac{1}{625}(-49500n^7-229500n^6+9823185n^5-58733106n^4\\
&&~~~~~~~~~~~~~~+167513583n^3-286491054n^2+305387712n-157410096)<0,
\end{eqnarray*}
\begin{eqnarray*}
&&h_{n,3}\Big(\frac{1}{2}(n-0.9)\sqrt{n-5}\Big)\\
&&~~~~~~~~~~=\frac{1}{625}(130500n^7-2596500n^6+19697985n^5 - 77235816n^4 \\
&&~~~~~~~~~~~~~~+ 185131179n^3 - 298187244n^2 + 312906240n -160065600)>0.
\end{eqnarray*}
 It then follows that $\rho_{ex}(D_{n,1})<\frac{1}{2}(n-0.9)\sqrt{n-5}$.

For other $D_{n,i}$ with $2\leq i\leq9$, we can obtain $\rho_{ex}(D_{n,i})<\frac{1}{2}(n-0.9)\sqrt{n-5}$ for $n\geq12$ in an analogous way. The details are omitted here.
\end{proof}

Now we deal with all the graphs of order $n\geq 12$ with maximum degree $\Delta\leq n-3$ in a unified way as follows.

\begin{lemma}\label{l4.7}
Let $G$ be a bicyclic graph of order $n\geq12$ with maximum degree at most $n-3$. Then
$\rho_{ex}(G)<\rho_{ex}(G_2)$.
\end{lemma}
\begin{proof}
For $\Delta=n-3$, by Lemma \ref{l4.1}, $\rho_{ex}(G)
\leq\frac{1}{2}\Big(n-3+\frac{1}{n-3}\Big)\rho(G)$ since function $g(x)=x+\frac{1}{x}$ is increasing in $x$ when $x\geq 1$. And by Lemma \ref{l4.2},  $\rho(G)\leq \rho(H(n, n-3, 2))$. And hence, we have
\begin{eqnarray*}
\rho_{ex}(G)
\leq\frac{1}{2}\Big(n-3+\frac{1}{n-3}\Big)\rho(H(n, n-3, 2))
\leq\frac{1}{2}\Big(n-3+\frac{1}{n-3}\Big)\sqrt{n}
\end{eqnarray*}
 for $12\leq n \leq20$ , by Lemma \ref{l4.4}.
The values of $\frac{1}{2}(n-3+\frac{1}{n-3})\sqrt{n}$ and $\rho_{ex}(G_2)$ for  $12\leq n\leq20$ are shown in Table \ref{tabExtended}. We can obtain $\rho_{ex}
(G)\leq\rho_{ex}(G_2)$ for $12\leq n\leq20$.

Suppose in the following $n \geq20$, similarly, we have
\begin{eqnarray*}
\rho_{ex}(G)
\leq\frac{1}{2}\Big(n-3+\frac{1}{n-3}\Big)\rho(H(n, n-3, 2))
\leq\frac{1}{2}\Big(n-3+\frac{1}{n-3}\Big)\sqrt{n-1.2}.
\end{eqnarray*}
by Lemmas \ref{l4.1}, \ref{l4.2} and \ref{l4.4}.
Furthermore, it's easy to check that
$$\frac{1}{2}\Big(n-3+\frac{1}{n-3}\Big)\sqrt{n-1.2}<\frac{1}{2}(n-0.9)\sqrt{n-5}$$
when $n\geq20$.

Therefore we have $\rho_{ex}(G)<\rho_{ex}(G_2)$ for all $n\geq12$ by Lemma \ref{l4.5}.

Next, we consider all the graphs of order $n\geq12$ with maximum degree  $\Delta \leq n-4$. By Lemma \ref{l4.1}, $\rho_{ex}(G)\leq\frac{1}{2}(\Delta+\frac{1}{\Delta})\rho(G)
\leq\frac{1}{2}(n-4+\frac{1}{n-4})\rho(G)$ since function $g(x)=x+\frac{1}{x}$ is increasing in $x$ when $x\geq 1$, and by Lemmas \ref{l4.2} and \ref{l4.3},  $\rho(G)\leq \rho(H(n, n-3, 2))$.  And hence, we can deduce that
\begin{eqnarray*}
\rho_{ex}(G)
\leq\frac{1}{2}\Big(n-4+\frac{1}{n-4}\Big)\rho(H(n, n-3, 2))
\leq\frac{1}{2}\Big(n-4+\frac{1}{n-4}\Big)\sqrt{n}.
\end{eqnarray*}
for $n\geq 12$, by Lemma \ref{l4.4}.
Moreover
$$\frac{1}{2}\Big(n-4+\frac{1}{n-4}\Big)\sqrt{n}<\frac{1}{2}(n-0.9)\sqrt{n-5}$$
when $n\geq 8$.

Therefore, we get the desired result by Lemma \ref{l4.5}, completing the proof of this lemma.
\end{proof}

\begin{table}[h]
\centering
\caption{\newline The values of $\frac{1}{2}(n-3+\frac{1}{n-3})\sqrt{n}$ and $\rho_{ex}(G_2)$ for  $12\leq n\leq20$.}\label{tabExtended}
\begin{tabular}{|c|c|c|c|c|c|}
\hline      $n$ & 12&	13&	14&	15&	16\\
\hline  $\frac{1}{2}(n-3+\frac{1}{n-3})\sqrt{n}$& 15.7809&	18.208&	20.7492&	23.3993&	26.1538	\\
\hline   $\rho_{ex}(G_2)$&	15.8028&	18.2277&	20.7672&	23.4160&	26.1695\\
\hline
\end{tabular}
\begin{tabular}{|c|c|c|c|c|c|}
\hline     $n$  & 17&	18&	19&	20& \quad\quad\,\quad\,\\
\hline  $\frac{1}{2}(n-3+\frac{1}{n-3})\sqrt{n}$& 29.009&	31.9612&	35.0074&	38.1447& \quad\\
\hline   $\rho_{ex}(G_2)$& 29.0238&	31.9753&	35.0209&	38.1576& \\
\hline
\end{tabular}
\end{table}

\noindent{\bf The proof of Theorem \ref{th4.1}. }
Combining the  Lemmas \ref{l4.5}, \ref{l4.6} and \ref{l4.7}, the theorem holds naturally.
\end{proof}

{\bf Acknowledgements}
This work was supported by the National Natural Science Foundation of China [Grant number: 11971406].

\clearpage 
\section*{Appendix}

\begin{table}[htb]
\centering
\caption{The approximate values of $\rho(A_f(G_i))$} ($i=2,3,4$) for $n=6$\\  \label{tabn=6}
\begin{tabular}{|c|c|c|c|c|}
\hline      f(x,y)           &    $1$ & $x+y$    & $(x+y)^2$ & $(x+y)^3$\\
\hline   $\rho(A_f(G_2))$ & 2.7039 & 17.0855  & 111.8198  &749.14\\
\hline   $\rho(A_f(G_3))$ & 2.7321 & 16.3940  & 101.8670  &652.82\\
\hline   $\rho(A_f(G_4))$ &{\bf2.7913} &{\bf17.6015} &{\bf114.6620} &{\bf788.49} \\
\hline
\end{tabular}
\end{table}

\begin{table}[htb]
\centering
\caption{The approximate values of $\rho(A_f(G_i))$} ($i=2,3,4$) for $n=7$\\  \label{tabn=7}
\begin{tabular}{|c|c|c|c|c|}
\hline      f(x,y)           &    $1$    & $x+y$    & $(x+y)^2$ & $(x+y)^3$\\
\hline   $\rho(A_f(G_2))$ &2.8558& {\bf20.4063}&{\bf152.0299}  &{\bf1159.8}\\
\hline   $\rho(A_f(G_3))$ & 2.8332    & 18.6430    & 128.7889  &926.19\\
\hline   $\rho(A_f(G_4))$ & {\bf 2.9032}    &20.0004    &143.5387 &1131 \\
\hline
\end{tabular}
\end{table}

\end{document}